\documentclass[11pt, a4paper]{article}
\usepackage{amsfonts, amsmath, amssymb, latexsym, mathrsfs, amsthm, authblk, cleveref, tikz,amscd,graphicx, xcolor, enumerate, pifont}
\theoremstyle{definition}
\newtheorem{theorem}{Theorem}[section]
\newtheorem{lemma}[theorem]{Lemma}

\newtheorem{definition}[theorem]{Definition}

\newtheorem{notation}[theorem]{Notation}

\numberwithin{equation}{section}

\newcommand{\CE}{\mathcal{E}}
\newcommand{\CR}{\mathcal{R}}

\newcommand{\CY}{\mathcal{Y}}
\newcommand{\CU}{\mathcal{U}}

\newcommand{\CP}{\mathcal{P}}
\newcommand{\CA}{\mathcal{A}}

\renewcommand{\CD}{\mathcal{D}}

\newcommand{\CH}{\mathcal{H}}

\newcommand{\CX}{\mathcal{X}}

\newcommand{\CV}{\mathcal{V}}
\newcommand{\CB}{\mathcal{B}}

\newcommand{\BR}{\mathbb{R}}

\newcommand{\BZ}{\mathbb{Z}}

\newcommand{\BN}{\mathbb{N}}

\newcommand{\abs}[1]{\left|#1\right|}
\newcommand{\fl}[1]{\left\lfloor#1\right\rfloor}
\newcommand{\ce}[1]{\left\lceil#1\right\rceil}
\newcommand{\gs}{\geqslant}
\newcommand{\ls}{\leqslant}
\newcommand{\pt}[1]{\left(#1\right)}
\newcommand{\cb}[1]{\left\{#1\right\}}

\begin{document}
\setcounter{page}{1}

\title{\bf Hilton-Milner theorem for $k$-multisets}

\author[1]{
Jiaqi Liao\thanks{E-mail:\texttt{975497560@qq.com}}}
\author[1]{
Zequn Lv\thanks{E-mail:\texttt{lvzq19@mails.tsinghua.edu.cn}}}
\author[2]{
Mengyu Cao\thanks{Corresponding author. E-mail:\texttt{myucao@ruc.edu.cn}}}
\author[1]{
Mei Lu\thanks{E-mail:\texttt{lumei@mail.tsinghua.edu.cn}}}

\affil[1]{\small Department of Mathematical Sciences, Tsinghua University, Beijing 100084, China}
\affil[2]{\small Institute for Mathematical Sciences, Renmin University of China, Beijing 100086, China}

\date{}
\openup 0.5\jot
\maketitle

\begin{abstract}
	
Let $ k, n \in \mathbb{N}^+ $ and $ m \in \mathbb{N}^+ \cup \{\infty \} $. A $ k $-multiset in $ [n]_m $ is a $ k $-set whose elements are integers from $ \{1, 2, \ldots, n\} $, and each element is allowed to have at most $ m $ repetitions. A family of $ k $-multisets in $ [n]_m $ is said to be intersecting if every pair of $ k $-multisets from the family have non-empty intersection. In this paper, we give the size and structure of the largest non-trivial intersecting family of $ k $-multisets in $ [n]_m $ for $ n \gs k + \ce{k/m} $. In the special case when $m=\infty$, our result gives rise to an unbounded multiset version for Hilton-Milner Theorem given by Meagher and Purdy.  Furthermore, our main theorem unites the statements of the Hilton-Milner Theorem for finite sets and unbounded multisets.
	
\vspace{2mm}
	
	
\
	
\end{abstract}

\section{I{ntroduction}}\label{sec1}

This paper is devoted to the study of non-trivial intersecting families in bounded multisets. After some preliminaries we pass to the proof of our main result. We begin by discussing some elementary concepts which are basic to the theory developed below.

We shall use the following notation. Let $ \BN = \cb{0, 1, 2, \ldots} $ be the set of natural numbers. Let $ k, n \in \BN^+ $, $ m \in \BN^+ \cup \cb{\infty} $, and $ [n]_m := \cb{m \cdot 1, m \cdot 2, \ldots, m \cdot n} $, that is, $ [n]_m $ contains exactly $ m $ symbols $ i $ for each $ i = 1, 2, \ldots, n $. Then $[n]_1=[n]= \cb{1, 2, \ldots, n} $. In particular, when $ m \ne \infty $, $ [n]_m $ is said to be \emph{bounded}, and  $ [n]_\infty $ is called \emph{unbounded}. The \emph{cardinality} of
$ A = \cb{\mu_1 \cdot 1, \mu_2 \cdot 2, \ldots, \mu_n \cdot n}$ is denoted by $ \abs{A} $ and defined by $ \sum\limits_{i = 1}^{n}\mu_i $. If $A' = \cb{\mu_1' \cdot 1, \mu_2' \cdot 2, \ldots, \mu_n' \cdot n},$ then the \emph{intersection} of $ A $ and $ A' $ is denoted by $ A \cap A' $ and defined by $\cb{\min\cb{\mu_1, \mu_1'}\cdot 1, \min\cb{\mu_2, \mu_2'} \cdot 2, \ldots, \min\cb{\mu_n, \mu_n'} \cdot n}.$ The family of $ k $-\emph{uniform} subsets of $ [n]_m $ is denoted by $ \binom{[n]_m}{k} $ and defined by $\cb{A \subseteq [n]_m: \abs{A} = k}.$ Note that for any $ m \gs k $, we have $ \binom{[n]_m}{k} = \binom{[n]_\infty}{k} $. For $ \CA \subseteq [n]_m $, the \emph{total intersection} of $ \CA $ is denoted by $ \cap\CA $ and defined by $ \bigcap\limits_{A \in \CA} A $. We make the convention that $ \cap\emptyset \ne \emptyset $.

\begin{definition}\label{def}
	
	Let $ \CA \subseteq \binom{[n]_m}{k} $.
	
	\begin{enumerate}[(1)]
		\item $ \CA $ is \emph{intersecting} if for any $ A_1, A_2 \in \CA $, we have $ A_1 \cap A_2 \ne \emptyset $.
		
		\item $ \CA $ is \emph{maximal intersecting} with respect to $ \binom{[n]_m}{k} $ if $ \CA $ is not only intersecting, but also for any $ X \in \binom{[n]_m}{k} - \CA $, there exists $ A \in \CA $, such that $ A \cap X = \emptyset $.
		
		\item $ \CA $ is \emph{trivial intersecting} if $ \cap\CA \ne \emptyset $, and $ \CA $ is \emph{non-trivial intersecting} if $ \cap\CA = \emptyset $.
		
		\item A \emph{permutation} of $ [n]_1 $ is a bijective map from $ [n]_1 $ to itself. If
		$$ A = \cb{\mu_1 \cdot 1, \mu_2 \cdot 2, \ldots, \mu_n \cdot n} \subseteq [n]_m,$$
		then
		$ \sigma(A) := \cb{\mu_1 \cdot \sigma(1), \mu_2 \cdot \sigma(2), \ldots, \mu_n \cdot \sigma(n)} $
		is again a subset of $ [n]_m $ with the same cardinality as $ A $.
		
		\item Let $ \CA' \subseteq \binom{[n]_m}{k} $. $ \CA' $ is \emph{isomorphic to} $ \CA $ if there exists a permutation $ \sigma $ of $ [n]_1 $, such that $ \CA' = \cb{\sigma(A): A \in \CA} $. This relation is denoted by $ \CA' \cong \CA $.
	\end{enumerate}	
	
\end{definition}

Let $ a, b \in \BR $. The \emph{discrete interval} from $ a $ to $ b $ is denoted by $ [a, b] $ and defined by $ \cb{x \in \BZ: a \ls x \ls b} $. Denote $ H := [n - k + 1, n] $. Note that $ H $ possesses this property that $ \abs{H} = k $, $ 1 \notin H $ and $ H \subseteq [n]_1 $. The following two families are the most familiar structures in the field of we are studying.

\begin{enumerate}[(1)]
	\item The \emph{EKR-family} $ \CE_{n, k}^{m} := \cb{A \in \binom{[n]_m}{k}: 1 \in A} $, which is the candidate for the maximum intersecting family.
	
	\item The \emph{HM-family} $ \CH_{n, k}^{m} := \cb{A \in \CE_{n, k}^{m}: A \cap H \ne \emptyset} \cup \cb{H} $, which is the candidate for the non-trivial maximum intersecting family.
\end{enumerate}

\subsection{Background}

The Erd\H{o}s-Ko-Rado theorem is the fundamental result in extremal set theory that gives the size and structure of the largest intersecting family in $ \binom{[n]_1}{k} $.

\begin{theorem}[Erd\H{o}s-Ko-Rado theorem \cite{MR0140419}]\label{EKR61}
	
	Let $ n \gs 2k $ and $ \CA \subseteq \binom{[n]_1}{k} $. If $ \CA $ is intersecting,  then $ \abs{\CA} \ls \binom{n - 1}{k - 1} $. Moreover, if $ n > 2k $, then $ \abs{\CA} = \binom{n - 1}{k - 1} $ if and only if $ \CA \cong \CE_{n, k}^{1} $.
	
\end{theorem}

Since then, a lot of related works were motivated by the Erd\H{o}s-Ko-Rado theorem, many variants of Theorem \ref{EKR61} were shown up, see \cite{MR3534067, MR3497070} for more details. In \cite{MR0140419}, Erd\H{o}s, Ko and Rado also showed that any maximum $ t $-intersecting family is isomorphic to $ \cb{A \in \binom{[n]_1}{k}: [t]_1 \subseteq A} $ for $ n \gg 0 $. It is known that the threshold is $ (t + 1)(k - t + 1) $. This was proved by Frankl \cite{MR0519277} for $ t \gs 15 $ and later determined by Wilson \cite{MR0771733} for all $ t $. In \cite{MR0519277}, Frankl conjectured on the maximum size of $ t $-intersecting family for all triple of positive integers $ k, n, t $. This conjecture was partially solved by Frankl and F\"uredi in \cite{MR1092847} and completely settled by Ahlswede and Khachatrian in \cite{MR1429238}.

Ascertaining the size and structure of non-trivial maximum intersecting family in $ \binom{[n]_1}{k} $ was an open problem for a long time. The first result was the following, for $ t = 1 $.

\begin{theorem}[Hilton-Milner Theorem \cite{MR0219428}]\label{HM67}
	
	Let $ n \gs 2k \gs 6 $ and $ \CA \subseteq \binom{[n]_1}{k} $. If $ \CA $ is non-trivial intersecting, then $ \abs{\CA} \ls \binom{n - 1}{k - 1} - \binom{n - k - 1}{k - 1} + 1 $. Moreover, if $ n > 2k > 6 $, then $ \abs{\CA} = \binom{n - 1}{k - 1} - \binom{n - k - 1}{k - 1} + 1 $ if and only if $ \CA \cong \CH_{n, k}^{1} $.
	
\end{theorem}

In \cite{MR0480051, MR0585195}, Frankl dealt with the case $ t \gs 2 $ for $ n \gg 0 $. In \cite{MR0826944}, Frankl and F\"uredi used the shifting technique to give an elegant proof of Theorem \ref{HM67}. Again, Ahlswede and Khachatrian completely solved this problem in \cite{MR1405994}. Recently, other non-trivial maximal $ t $-intersecting families with large size had been studied, see \cite{MR4275621, MR3565361, MR3626491}.

In this paper, we generalize Theorem \ref{HM67} to bounded multisets. Meagher and Purdy were the first to study intersecting families of $ \binom{[n]_\infty}{k} $. They obtained the Erd\H{o}s-Ko-Rado theorem for $ \binom{[n]_\infty}{k} $ in \cite{MR2861399}. In \cite{MR4717700}, the Erd\H{o}s-Ko-Rado theorem for $ \binom{[n]_m}{k} $ was obtained. Later, F\"uredi, Gerbner and Vizer obtained the $ t $-intersecting version of the Erd\H{o}s-Ko-Rado theorem for $ \binom{[n]_\infty}{k} $ in \cite{MR3339026}. Shortly after that, Meagher and Purdy obtained the $ t $-intersecting version of the Hilton-Milner theorem, but we only state the case $ t = 1 $ here.

\begin{theorem}[\cite{MR3425970}]\label{MP16}
	
	Let $ n \gs k + 1 \gs 4 $ and $ \CA \subseteq \binom{[n]_\infty}{k} $. If $ \CA $ is non-trivial intersecting, then $ \abs{\CA} \ls \binom{n + k - 2}{k - 1} - \binom{n - 2}{k - 1} + 1 $. Moreover, if $ n > k + 1 > 4 $, then $ \abs{\CA} = \binom{n + k - 2}{k - 1} - \binom{n - 2}{k - 1} + 1 $ if and only if $ \CA \cong \CH_{n, k}^{\infty} $.
	
\end{theorem}

We describe their method briefly. A Kneser graph $ K(n, k) $ is a graph with vertex set $ \binom{[n]_1}{k} $, and two vertices are adjacent if and only if the corresponding $ k $-sets are disjoint. Thus an independent set of $ K(n, k) $ is equivalent to an intersecting family of $ \binom{[n]_1}{k} $. Meagher and Purdy defined the graph $ M(n, k) $ with vertex set $ \binom{[n]_\infty}{k} $, and two vertices are adjacent if and only if the corresponding $ k $-multisets are disjoint. They constructed a bijective homomorphism $ f: V(K(n + k - 1, k)) \rightarrow V(M(n, k)) $, which induces the following inequality$$\alpha(M(n, k)) \ls \alpha(K(n + k - 1, k)).$$ Such a homomorphism exists, due to the fact that $ \abs{\binom{[n]_\infty}{k}} = \binom{n + k - 1}{k} $ is also a binomial coefficient. Unfortunately, $ \binom{[n]_m}{k} $ is not a binomial coefficient in general. But the principle is the same, that is, reduce to the case $ m = 1 $.

\subsection{Main result}

In \cite{MR4717700}, we proved the Erd\H{o}s-Ko-Rado theorem for $ \binom{[n]_m}{k} $. In this paper, we concern the Hilton-Milner theorem for $ \binom{[n]_m}{k} $. Throughout this paper we fix $ q := \ce{k/m} $ ($ \ce{k/\infty} := 1 $).

\begin{theorem}\label{sec1_main_thm}
	
	Let $ k \gs 4 $, $ 2 \ls m \ls \infty $, $ n \gs k + q $ and $ \CA \subseteq \binom{[n]_m}{k} $. If $ \CA $ is non-trivial intersecting, then $ \abs{\CA} \ls \abs{\CH_{n, k}^{m}} $. Moreover, if one of the following two conditions hold
	
	\begin{enumerate}[(A)]
		\item $ n > k + q $,
		
		\item $ n = k + q $ and $ \min\cb{k, m} \nmid k $,
	\end{enumerate}
	
	\noindent then $ \abs{\CA} = \abs{\CH_{n, k}^{m}} $ if and only if $ \CA \cong \CH_{n, k}^{m} $.
	
\end{theorem}

In the special case when $m=\infty$, our result gives rise to an unbounded multiset version for Hilton-Milner Theorem given by Meagher and Purdy (Theorem~\ref{MP16}). Furthermore, our main theorem unites the statements of the Hilton-Milner Theorem for finite sets and unbounded multisets.

The family of non-empty proper subsets of $ [n]_1 $ is denoted by $ \CP $ and defined by $ \cb{B: \emptyset \subsetneqq B \subsetneqq [n]_1} $.

\section{Intersecting families in $ \CP $}\label{sec2}

In this section, we give some more notation and lemmas which will be used in our proof of the main result. Throughout this paper we fix $ H := [n - k + 1, n] $.

\begin{notation}
	
	Let $ k, n \in \BN^+ $, $ m \in \BN^+ \cup \cb{\infty} $ and $ \CB \subseteq \CP $.
	
	\begin{enumerate}[(1)]
		
		\item For $ B \in \CP $, the \emph{complement} of $ B $ is denoted by $ B^c $ and defined by $ [n]_1 - B $.
		
		\item The \emph{dual} of $ \CB $ is denoted by $ \CB^c $ and defined by $ \cb{B^c: B \in \CB} $.
		
		\item The $ i $-\emph{uniform part} of $ \CB $ is denoted by $ \CB(i) $ and defined by $ \cb{B \in \CB: \abs{B} = i} $.
		
		\item The \emph{valuable part} of $ \CB $ is denoted by $ \CB^\star $ and defined by $ \bigcup\limits_{i = q}^{k} \CB(i) $.
		
		\item The family corresponding to $ \CE_{n, k}^{m} $ is denoted by $ \CU $ and defined by $ \cb{B \in \CP: 1 \in B} $.
		
		\item The \emph{removed part} of $ \CU $ is denoted by $ \CR $ and defined by $ \cb{B \in \CU: B \subseteq [n - k]_1} $. Note that $ H^c = [n - k]_1 $.
		
		\item The family corresponding to $ \CH_{n, k}^{m} $ is denoted by $ \CV $ and defined by $ (\CU - \CR) \cup \CR^c $. Note that $ \CV^\star = \cb{B \in \bigcup\limits_{i = q}^{k} \CU(i): B \cap H \ne \emptyset} \cup \cb{H} $.
	\end{enumerate}
	
\end{notation}

The first step is to evaluate the size of the $ \ell $-uniform part of $ \CV $.

\begin{lemma}\label{computation}
	
	Let $ k \gs 4 $, $ 2 \ls m \ls \infty $ and $ n \gs k + q $. If $ 2 \ls \ell \ls w := \min\cb{k, \fl{n/2}} $, then $ \abs{\CV(\ell)} \gs \binom{n - 1}{\ell - 1} - \binom{n - \ell - 1}{\ell - 1} + 1 $. Moreover, if $ \ell \ne w $, then $ \abs{\CV(\ell)} > \binom{n - 1}{\ell - 1} - \binom{n - \ell - 1}{\ell - 1} + 1 $.
	
\end{lemma}

\begin{proof}
	
	Recall the definition of $ \CV $, which is $ (\CU - \CR) \cup \CR^c $, where $ \CR = \cb{B \in \CU: B \subseteq [n - k]_1} $. Note that
	\begin{align*}
		\abs{\CV(\ell)}
		&= \abs{\CU(\ell)} - \abs{\CR(\ell)} + \abs{\CR^c(\ell)} \\
		&= \abs{\CU(\ell)} - \abs{\CR(\ell)} + \abs{\CR(n - \ell)^c} \\
		&= \abs{\CU(\ell)} - \abs{\CR(\ell)} + \abs{\CR(n - \ell)}.
	\end{align*}
	We discuss in two cases.
	
	\begin{enumerate}[{Case} A:]
		\item $ n \gs 2k $.
		
		In this case, we have $ w = k $ and so $ 2 \ls \ell \ls k $. Consider the non-zero interval of $ \abs{\CR(\ell)} $ and $ \abs{\CR(n - \ell)} $ respectively.
		
		\begin{enumerate}[(a)]
			\item $ \abs{\CR(\ell)} \ne 0 \Leftrightarrow \ell \in [1, n - k] \cap [2, k] = [2, k] $.
			
			\item $ \abs{\CR(n - \ell)} \ne 0 \Leftrightarrow n - \ell \in [1, n - k] \cap [n - k, n - 2] = \cb{n - k} $, i.e., $ \ell = k $.
		\end{enumerate} Then\begin{equation}\label{eqn1}
			\abs{\CV(\ell)} = \left\{\begin{aligned}
				&\binom{n - 1}{\ell - 1} - \binom{n - k - 1}{\ell - 1},&\qquad&\text{if } 2 \ls \ell < k,\\
				&\binom{n - 1}{k - 1} - \binom{n - k - 1}{k - 1} + 1,&\qquad&\text{if } \ell = k.
			\end{aligned}\right.
		\end{equation}
		
		\noindent Hence $ \abs{\CV(\ell)} \gs \binom{n - 1}{\ell - 1} - \binom{n - \ell - 1}{\ell - 1} + 1 $ when $ n \gs 2k $.
		
		\item $ n < 2k $.
		
		In this case, we have $ w = \fl{n/2} $ and so $ 2 \ls \ell \ls \fl{n/2} $. Consider the non-zero interval of $ \abs{\CR(\ell)} $ and $ \abs{\CR(n - \ell)} $ respectively.
		
		\begin{enumerate}[(a)]
			\item $ \abs{\CR(\ell)} \ne 0 \Leftrightarrow \ell \in [1, n - k] \cap [2, \fl{n/2}] = [2, n - k] $.
			
			\item $ \abs{\CR(n - \ell)} \ne 0 \Leftrightarrow n - \ell \in [1, n - k] \cap [\ce{n/2}, \fl{n/2}] = \emptyset $, i.e., $ \abs{\CR(n - \ell)} \equiv 0 $.
		\end{enumerate} Then\begin{equation}\label{eqn2}
			\abs{\CV(\ell)} = \left\{\begin{aligned}
				&\binom{n - 1}{\ell - 1} - \binom{n - k - 1}{\ell - 1},&\qquad&\text{if } 2 \ls \ell < k,\\
				&\binom{n - 1}{\ell - 1},&\qquad&\text{if } \ell = k.
			\end{aligned}\right.
		\end{equation}
		
		\noindent Hence $ \abs{\CV(\ell)} \gs \binom{n - 1}{\ell - 1} - \binom{n - \ell - 1}{\ell - 1} + 1 $ when $ n < 2k $.
		
		Moreover, assume $ \ell \ne w $, we need to show that $ \abs{\CV(\ell)} > \binom{n - 1}{\ell - 1} - \binom{n - \ell - 1}{\ell - 1} + 1 $. Otherwise, suppose $ \abs{\CV(\ell)} = \binom{n - 1}{\ell - 1} - \binom{n - \ell - 1}{\ell - 1} + 1 $. Since $ \ell \ne w $, the only possibility for $ \abs{\CV(\ell)} = \binom{n - 1}{\ell - 1} - \binom{n - \ell - 1}{\ell - 1} + 1 $ is $ \abs{\CV(\ell)} = \binom{n - 1}{\ell - 1} - \binom{n - k - 1}{\ell - 1} $. Thus
		\begin{align*}
			1
			&= \binom{n - \ell - 1}{\ell - 1} - \binom{n - k - 1}{\ell - 1} \\
			&\gs \binom{n - \ell - 1}{\ell - 1} - \binom{n - \ell - 2}{\ell - 1} \\
			&= \binom{n - \ell - 2}{\ell - 2} \\
			&\gs \binom{\ell - 1}{\ell - 2} \\
			&= \ell - 1.
		\end{align*}
		Since $ \ell \gs 2 $, we have $ \ell = 2 $ and then $ k = \ell + 1 = 3 $, contradict to $ k \gs 4 $. \qedhere
		
	\end{enumerate}
	
\end{proof}

From now on, we study the intersecting families in $ \CP $. The following definitions are analogs of Definition \ref{def}.

\begin{definition}
	
	Let $ \CB \subseteq \CP $.
	
	\begin{enumerate}[(1)]
		\item $ \CB $ is \emph{intersecting} if for any $ B_1, B_2 \in \CB $, we have $ B_1 \cap B_2 \ne \emptyset $.
		
		\item $ \CB $ is \emph{maximal intersecting} with respect to $ \CP $ if $ \CB $ is not only intersecting, but also for any $ Y \in \CP - \CB $, there exists $ B \in \CB $, such that $ B \cap Y = \emptyset $.
		
		\item Let $ \CB' \subseteq \CP $. $ \CB' $ is \emph{isomorphic to} $ \CB $ if there exists a permutation $ \sigma $ of $ [n]_1 $, such that $ \CB' = \cb{\sigma(B): B \in \CB} $. This relation is denoted by $ \CB' \cong \CB $.
	\end{enumerate}		
	
\end{definition}

The following Lemma tells us that any maximal intersecting family in $ \CP $ has cardinality $ 2^{n - 1} - 1 $.

\begin{lemma}[\cite{MR4717700}, Lemma 3.1 and Corollary 3.2]\label{MR4717700}
	
	Let $ \CB \subseteq \CP $ and $ \CB $ is maximal intersecting with respect to $ \CP $. Then
	
	\begin{enumerate}[(i)]
		\item for any $ B \in \CP $, we have $ \abs{\cb{B, B^c} \cap \CB} = 1 $.
		
		\item $ \CB $ is a $ \CP $-up-set (i.e., if $ B \in \CB $ and $ B \subseteq B' \in \CP $, then $ B' \in \CB $).
		
	\end{enumerate}	
	
\end{lemma}

The following two Lemmas is used in the proof of Lemma \ref{comparison}.

\begin{lemma}\label{down_set}
	
	Let $ \CB \subseteq \CP $ and $ \CB $ is maximal intersecting with respect to $ \CP $. Let $ \CD \subseteq \CU $. Then the following two statements are equivalent:
	
	\begin{enumerate}[(i)]
		\item $ (\CU - \CD) \cup \CD^c $ is maximal  intersecting with respect to $ \CP $.
		
		\item \begin{enumerate}[(a)]
			\item $ \CD $ is a $ \CU $-down-set (i.e., if $ D \in \CD $ and $ D \supseteq D' \in \CU $, then $ D' \in \CD $).
			
			\item If $ D_1, D_2 \in \CD $, then $ D_1 \cup D_2 \ne [n]_1 $.
		\end{enumerate}
	\end{enumerate}	
	
\end{lemma}

\begin{proof}
	
	We prove the necessity and sufficiency respectively.  Firstly, if (ii) holds, since $ \abs{(\CU - \CD) \cup \CD^c} = 2^{n - 1} - 1 $, we only need to show that $ (\CU - \CD) \cup \CD^c $ is intersecting. Note that $ \CU - \CD \subseteq \CU $, hence $ \CU - \CD $ is already intersecting. We discuss in two cases:
		
		\begin{enumerate}[{Case} A:]
			\item $ D^c \in \CD^c $ and $ B \in \CU - \CD $.
			
			In this case, $ D \in \CD $. By (ii)(a) and $ B \notin \CD $, we have $ B \not\subseteq D $ (otherwise $ B \in \CD $), so $ B \cap D^c \ne \emptyset $.
			
			\item $ D_1^c, D_2^c \in \CD^c $.
			
			In this case, $ D_1, D_2 \in \CD $. By (ii)(b), we have $ D_1^c \cap D_2^c = (D_1 \cup D_2)^c \ne \emptyset $.
		\end{enumerate}
Thus we have (i) holds.		
		
Secondly, if (i) holds, we prove (a) and (b) of (ii), respectively.
		
		\begin{enumerate}[(a)]
			\item Let $ D \in \CD $ and $ D \supseteq D' \in \CU $. Then $ D' \cap D^c = \emptyset $. Since $ (\CU - \CD) \cup \CD^c $ is intersecting, we have $ D' \notin (\CU - \CD) \cup \CD^c $. By Lemma \ref{MR4717700}(i), we have $ {D'}^c \in (\CU - \CD) \cup \CD^c $, so $ D' \in \CD $.
			
			\item Let $ D_1, D_2 \in \CD $. Then $ D_1^c, D_2^c \in \CD^c $. Since $ (\CU - \CD) \cup \CD^c $ is intersecting, we have $ \CD^c $ is also intersecting, thus $ D_1^c \cap D_2^c \ne \emptyset $, i.e., $ D_1 \cup D_2 \ne [n]_1 $.
		\end{enumerate}
Thus we have (ii) holds. We complete the proof of this lemma.\qedhere	
\end{proof}

\begin{lemma}\label{n_minus_k}
	
	Let $ k \gs 4 $, $ 2 \ls m \ls \infty $ and $ n \gs k + q $. If $ \CB \subseteq \CP $ and $ \CB $ is maximal intersecting with respect to $ \CP $ with $ \cap\CB^\star = \emptyset $, then $ (\CU - \CB)(n - k) \ne \emptyset $.
	
\end{lemma}

\begin{proof}
	
	Set $ \CD := \CU - \CB $. Suppose $ \CD(n - k) = \emptyset $. By Lemma \ref{down_set}, we have $ \CD $ is a $ \CU $-down-set, thus $ \CD(n - k + 1) = \cdots = \CD(n - 1) = \emptyset $. So$$\CD^c = \pt{\bigcup_{i = 1}^{n - k - 1}\CD(i)}^c = \bigcup_{i = 1}^{n - k - 1}\CD^c(n - i) = \bigcup_{i = k + 1}^{n - 1}\CD^c(i).$$
	Since $ [q, k] \cap [k + 1, n - 1] = \emptyset $, we have $ \pt{\CD^c}^\star = \bigcup\limits_{i = q}^{k}\CD^c(i) = \emptyset $. Then $ \CB^\star = \pt{\CU - \CD}^\star \cup \pt{\CD^c}^\star = \pt{\CU - \CD}^\star \subseteq \CU $, contradict to $ \cap\CB^\star = \emptyset $. \qedhere
	
\end{proof}

The following Lemma characterizes the layers of $ \CV $.

\begin{lemma}\label{comparison}
	
	Let $ k \gs 4 $, $ 2 \ls m \ls \infty $ and $ n \gs k + q $. Let $ \CB \subseteq \CP $ and $ \CB $ is maximal intersecting with respect to $ \CP $ with $ \cap\CB^\star = \emptyset $. If $ 2 \ls \ell \ls w := \min\cb{k, \fl{n/2}} $, then $ \abs{\CV(\ell)} \gs \abs{\CB(\ell)} $.
	
\end{lemma}

\begin{proof}
	
	Suppose there exists $ 2 \ls \ell_0 \ls w $, such that $ \abs{\CB(\ell_0)} > \abs{\CV(\ell_0)} $. By Lemma \ref{computation}, we have $ \abs{\CB(\ell_0)} > \binom{n - 1}{\ell_0 - 1} - \binom{n - \ell_0 - 1}{\ell_0 - 1} + 1 $. By Theorem \ref{HM67}, we have $ \cap\CB(\ell_0) \ne \emptyset $. Without loss of generality, we may assume $ 1 \in \cap\CB(\ell_0) $, i.e., $ \CB(\ell_0) \subseteq \CU(\ell_0) $. Set $ \CD := \CU - \CB $. Because $ \CB = \pt{\CU - \CD} \cup \CD^c $, we have $$\CB(\ell_0) = \CU(\ell_0) - \CD(\ell_0) \cup \CD(n - \ell_0)^c.$$ By $ \CB(\ell_0) \subseteq \CU(\ell_0) $, we have $ \CD(n - \ell_0) = \emptyset $. By Lemma \ref{n_minus_k}, we have $ \CD(n - k) \ne \emptyset $, thus $ \ell_0 < k $. Without loss of generality, we may assume $ [n - k]_1 \in \CD(n - k) $. By Lemma \ref{down_set}, we have $ \CD $ is a $ \CU $-down-set, hence $ \CR \subseteq \CD $. Then
	\begin{align*}
		\abs{\CB(\ell_0)}
		&= \abs{\CU(\ell_0)} - \abs{\CD(\ell_0)} \\
		&\ls \abs{\CU(\ell_0)} - \abs{\CR(\ell_0)} \\
		&= \binom{n - 1}{\ell_0 - 1} - \binom{n - k - 1}{\ell_0 - 1} \\
		&\ls \abs{\CV(\ell_0)},
	\end{align*} contradict to $ \abs{\CB(\ell_0)} > \abs{\CV(\ell_0)} $. \qedhere
	
\end{proof}

The following Lemma characterizes the layers of $ \CV $ a step further.

\begin{lemma}\label{structure}
	
	Let $ k \gs 4 $, $ 2 \ls m \ls \infty $ and $ n \gs k + q $. If $ \CB \subseteq \CP $ and $ \CB $ is maximal intersecting with respect to $ \CP $ with $ \cap\CB^\star = \emptyset $, then $ \abs{\CB(q)} = \abs{\CV(q)} $ only if $ \CB^\star \cong \CV^\star $.	
	
\end{lemma}

\begin{proof}
	
	Recall that $ H := [n - k + 1, n] $. Set $ \CD := \CU - \CB $ and $ w := \min\cb{k, \fl{n/2}} $. By Lemma \ref{computation} and $ q < w $, we have$$\abs{\CB(q)} = \abs{\CV(q)} > \binom{n - 1}{q - 1} - \binom{n - q - 1}{q - 1} + 1.$$ By Theorem \ref{HM67}, we have $ \cap\CB(q) \ne \emptyset $. Without loss of generality, we may assume $ 1 \in \cap \CB(q) $. By Lemma \ref{n_minus_k}, we have $ \CD(n - k) \ne \emptyset $, and note that $ \CB = (\CU - \CD) \cup \CD^c $, we have$$\CB(k) \supseteq \CD^c(k) = \CD(n - k)^c.$$	Wituout loss of generality, we may assume $ H \in \CB(k) $. Because $ \CB $ is maximal intersecting with respect to $ \CP $, we have$$\CB(q) \subseteq \CV(q) = \cb{B \in \CU(q): B \cap H \ne \emptyset}.$$ Hence $ \CB(q) = \CV(q) $.
	
	Let $ q < \ell \ls w $. We need to show that $ \CB(\ell) = \CV(\ell) $. We discuss in two cases:
	
	\begin{enumerate}[{Case} A:]
		\item $ n \gs 2k $.
		
		In this case, we have $ w = k $ and so $ q < \ell \ls k $. Note that any member of$$\cb{B \in \CU(\ell): B \cap H \ne \emptyset}$$ is a superset of certain member of $ \CB(q) $. By Lemma \ref{MR4717700}(ii), we have $ \CB $ is a $ \CP $-up-set, then$$\CB(\ell) \supseteq \left\{\begin{aligned}
			&\cb{B \in \CU(\ell): B \cap H \ne \emptyset} =: S_1,&\qquad&\text{ if } \ell < k,\\
			&\cb{B \in \CU(k): B \cap H \ne \emptyset} \cup \cb{H} =: S_2,&\qquad&\text{ if } \ell = k.
		\end{aligned}\right.$$
		
		\noindent By (\ref{eqn1}), we have $ \abs{S_i} = \abs{\CV(\ell)} $ ($ i = 1, 2 $). Note that $ \abs{\CB(\ell)} \gs \abs{S_i} $. By Lemma \ref{comparison}, we have $ \abs{\CV(\ell)} \gs \abs{\CB(\ell)} $. Thus $ \abs{\CB(\ell)} = \abs{S_i} $, which implies that $ \CB(\ell) = \CV(\ell) $.
		
		\item $ n < 2k $.
		
		In this case, we have $ w = \fl{n/2} $ and so $ q < \ell \ls \fl{n/2} $. Note that any member of$$\left\{\begin{aligned}
			&\cb{B \in \CU(\ell): B \cap H \ne \emptyset},&\qquad&\text{ if } \ell \ls n - k,\\
			&\CU(\ell),&\qquad&\text{ if } n - k < \ell \ls \fl{n/2}.
		\end{aligned}\right.$$ is a superset of certain member of $ \CB(q) $. By Lemma \ref{MR4717700}(ii), we have $ \CB $ is a $ \CP $-up-set, then$$\CB(\ell) \supseteq \left\{\begin{aligned}
			&\cb{B \in \CU(\ell): B \cap H \ne \emptyset} =: S_3,&\qquad&\text{ if } \ell \ls n - k,\\
			&\CU(\ell) =: S_4,&\qquad&\text{ if } n - k < \ell \ls \fl{n/2}.
		\end{aligned}\right.$$
		
		\noindent By (\ref{eqn2}), we have $ \abs{S_i} = \abs{\CV(\ell)} $ ($ i = 3, 4 $). Note that $ \abs{\CB(\ell)} \gs \abs{S_i} $. By Lemma \ref{comparison}, we have $ \abs{\CV(\ell)} \gs \abs{\CB(\ell)} $. Thus $ \abs{\CB(\ell)} = \abs{S_i} $, which implies that $ \CB(\ell) = \CV(\ell) $. \qedhere
	\end{enumerate}
	
\end{proof}
\section{Proof of the main theorem}\label{sec3}

\begin{notation}
	
	Fix $ m \in \BN^+ \cup \cb{\infty} $. The coefficient of $ x^k $ in the polynomial $ \pt{\sum\limits_{i = 1}^{m} x^i}^\ell $ is denoted by $ C_{k, \ell} $.	
	
\end{notation}

Recall that $ q := \ce{k/m} $.

\begin{lemma}[\cite{MR4717700}, Lemma 2.1 and Corollary 2.8]\label{key}
	
	$ C_{k, \ell} $ satisfy the following four  properties.
	
	\begin{enumerate}[(i)]
		\item $ C_{k, \ell} > 0 $ if and only if $ q \ls \ell \ls k $.
		
		\item $ C_{k, q} = 1 $ if and only if $ \min\cb{k, m} \mid k $.
		
		\item $ C_{k, k} \equiv 1 $.
		
		\item If $ n \gs k + q $ and $ q \ls \ell \ls \min\cb{k, \fl{n/2}} $, then $ C_{k, \ell} \gs C_{k, n - \ell} $.
	\end{enumerate}	
	
\end{lemma}

\begin{notation}
	
	Let $ k, n \in \BN^+ $, $ m \in \BN^+ \cup \cb{\infty} $ and $ A \in \binom{[n]_m}{k} $.
	
	\begin{enumerate}[(1)]
		\item The \emph{support} of $ A $ is denoted by $ \varphi(A) $ and defined by $ A \cap [n]_1 $. Note that $ \varphi $ is a map from $ \binom{[n]_m}{k} $ to $ \CP $.
		
		\item The \emph{preimage} of $ B $ is denoted by $ \varphi^{-1}(B) $ and defined by $ \cb{A \in \binom{[n]_m}{k}: \varphi(A) = B} $.
		
		\item The \emph{preimage} of $ \CB $ is denoted by $ \varphi^{-1}(\CB) $ and defined by $ \bigcup\limits_{B \in \CB} \varphi^{-1}(B) $.
	\end{enumerate}	
	
\end{notation}

\begin{lemma}
	
	Let $ k \in \BN^+ $, $ m \in \BN^+ \cup \cb{\infty} $ and $ n \gs k + q $. Then $ \CH_{n, k}^{m} $ is maximal intersecting with respect to $ \binom{[n]_m}{k} $.
	
\end{lemma}

\begin{proof}
	
	Recall that $ H := [n - k + 1, n] $. Note that $ \binom{[n]_k}{k} = \binom{[n]_\infty}{k} $, thus we may assume $ m \ls k $. It suffices to show that for any $ X \in \binom{[n]_m}{k} - \CH_{n, k}^{m} $, there exists $ A \in \CH_{n, k}^{m} $, such that $ A \cap X = \emptyset $. Since $ H \in \CH_{n, k}^{m} $ and $ X \notin \CH_{n, k}^{m} $, we may assume that $ X \cap H \ne \emptyset $ and $ 1 \notin X $. Set$$S := \{m\cdot i: i \notin X \cup \cb{1}\}.$$ Note that $ \abs{S} = m \cdot \pt{n - 1 - \abs{\varphi(X)}} $. We discuss in three cases:
	
	\begin{enumerate}[{Case} A:]
		\item $ m < k $. Now $ \abs{S} \gs m \cdot \pt{n - 1 - k} \gs m \cdot \pt{q - 1} \gs k - m \gs 1 $.
		
		\item $ m = k $ and $ n \gs k + 2 $. Now $ \abs{S} \gs m \cdot \pt{n - 1 - k} \gs m \gs 1 $.
		
		\item $ m = k $ and $ n = k + 1 $. Now $ \abs{\varphi(X)} < k $, hence $ \abs{S} \gs m \cdot \pt{n - 1 - k + 1} = m \gs 1 $.
	\end{enumerate}
	Since $ H \not\subseteq X $, we have $ H \cap S \ne \emptyset $, thus there exists $ A' \in \binom{S}{\max\cb{1, k - m}} $, such that $ A' \cap H \ne \emptyset $. Set $ A := \cb{\min\cb{m, k - 1}\cdot 1 }\cup A' $. Then $ A \cap X = \emptyset $.\qedhere
	
\end{proof}

The following two Lemmas relate the intersecting families in $ \binom{[n]_m}{k} $ and the intersecting families in $ \CP $.

\begin{lemma}[\cite{MR4717700}, {Lemma 3.3 and Lemma 3.4}]\label{relate}
	
	Let $ \CA \subseteq \binom{[n]_m}{k} $ and $ \CA $ is maximal intersecting with respect to $ \binom{[n]_m}{k} $. Let $ \CB \subseteq \CP $ and $ \CB $ is maximal intersecting with respect to $ \CP $. If $ \varphi(\CA) \subseteq \CB $, then $ \CA = \varphi^{-1}(\CB) = \sum\limits_{\ell = q}^{n - q} C_{k, \ell} \cdot \abs{\CB(\ell)} $.
	
\end{lemma}

\begin{lemma}\label{formula}
	
	Let $ \CA, \CX \subseteq \binom{[n]_m}{k} $ with $ \CA, \CX $ are maximal intersecting with respect to $ \binom{[n]_m}{k} $. Let $ \CB, \CY \subseteq \CP $ with $ \CB, \CY $ are maximal intersecting with respect to $ \CP $. If $ \varphi(\CA) \subseteq \CB $ and $ \varphi(\CX) \subseteq \CY $, then$$\abs{\CX} - \abs{\CA} = \sum_{\ell = q}^{w}\pt{C_{k, \ell} - C_{k, n - \ell}} \cdot \pt{\abs{\CY(\ell)} - \abs{\CB(\ell)}}.$$ where $ w := \min\cb{k, \fl{n/2}} $.	
	
\end{lemma}

\begin{proof}
	
	Since $ C_{k, \ell} = \abs{\varphi^{-1}([\ell]_1)} $, we have
	\begin{align*}
		\abs{\varphi^{-1}(\CB)}
		&= \abs{\varphi^{-1}\pt{\biguplus_{\ell = 1}^{n - 1}\CB(\ell)}}\\
		&= \sum_{\ell = 1}^{n - 1} \abs{\varphi^{-1}(\varOmega_\ell^1)} \cdot \abs{\CB(\ell)}\\
		&= \sum_{\ell = q}^{k} C_{k, \ell} \cdot \abs{\CB(\ell)}.
	\end{align*}
	
	By Lemma \ref{relate}, we have
	\begin{align*}
		\abs{\CX} - \abs{\CA}
		&= \abs{\varphi^{-1}(\CY)} - \abs{\varphi^{-1}(\CB)}\\
		&= \sum_{\ell = q}^{n - q} C_{k, \ell} \cdot \pt{\abs{\CY(\ell)} - \abs{\CB(\ell)}}\\
		&= \sum_{\ell = q}^{\fl{n/2}} C_{k, \ell} \cdot \pt{\abs{\CY(\ell)} - \abs{\CB(\ell)}} + \sum_{\ell = \ce{n/2}}^{n - q} C_{k, \ell} \cdot \pt{\abs{\CY(\ell)} - \abs{\CB(\ell)}}\\
		&= \sum_{\ell = q}^{\fl{n/2}} C_{k, \ell} \cdot \pt{\abs{\CY(\ell)} - \abs{\CB(\ell)}} - \sum_{\ell = q}^{\fl{n/2}} C_{k, n - \ell} \cdot \pt{\abs{\CY(\ell)} - \abs{\CB(\ell)}}\\
		&= \sum_{\ell = q}^{w} \pt{C_{k, \ell} - C_{k, n - \ell}} \cdot \pt{\abs{\CY(\ell)} - \abs{\CB(\ell)}}.\qedhere
	\end{align*}	
	
\end{proof}

\begin{proof}[Proof of Theorem\ref{sec1_main_thm}]
	
	Let $ \CB \subseteq \CP $ with $ \CB $ is maximal intersecting with respect to $ \CP $ and $ \varphi(\CA) \subseteq \CB $. By Lemma \ref{key} (iv) and \ref{formula}, we have$$\abs{\CH_{n, k}^{m}} - \abs{\CA} = \sum_{\ell = q}^{w}\pt{C_{k, \ell} - C_{k, n - \ell}} \cdot \pt{\abs{\CV(\ell)} - \abs{\CB(\ell)}} \gs 0.$$ Moreover, by Lemma \ref{key}(ii), if one of (A) and (B) holds, then $ C_{k, q} > C_{k, n - q} $. Hence $ \abs{\CB(q)} = \abs{\CV(q)} $. By Lemma \ref{structure}, we have $ \CB^\star \cong \CV^\star $, and so $ \CA \cong \CH_{n, k}^{m} $.\qedhere
	
\end{proof}
\noindent{\bf Acknowledgement.} M. Cao is supported by the National Natural Science Foundation of China (Grant 12301431), M. Lu is supported by the National Natural Science Foundation of China (Grant 12171272 \& 12161141003).

\bibliographystyle{abbrv}
\bibliography{References}
\addcontentsline{toc}{chapter}{Bibliography}

\end{document}